\documentclass[12pt, reqno]{amsart}
\usepackage{amscd}        
\usepackage{float}

\usepackage{nomencl}
\usepackage{url}

\usepackage{algorithm}
\usepackage{algpseudocode}
\usepackage{tabularx}

\usepackage{tabu}

\usepackage{multirow}

\usepackage{tikz-cd}
\usepackage[utf8]{inputenc}
\usepackage{mathtools}
\usepackage[utf8]{inputenc}
\usepackage{amsmath}
\usepackage{amssymb}
\usepackage{amsthm}
\usepackage[inner=1.0in,outer=1.0in,bottom=1.0in, top=1.0in]{geometry}
\DeclareFontFamily{OT1}{pzc}{}
\DeclareFontShape{OT1}{pzc}{m}{it}{<-> s * [1.10] pzcmi7t}{}
\DeclareMathAlphabet{\mathpzc}{OT1}{pzc}{m}{it}

\numberwithin{equation}{section}

\newcommand{\tpmod}[1]{{\@displayfalse\pmod{#1}}}

\numberwithin{equation}{section}

\newcommand{\lp}{\left(}
\newcommand{\rp}{\right)}

\newtheorem{theorem}{Theorem}[section]

\newtheorem{lemma}[theorem]{Lemma}
\newtheorem{proposition}[theorem]{Proposition}

\theoremstyle{definition}
\newtheorem{remark}[theorem]{Remark}

\title{Euler-type recurrences for $t$-color and $t$-regular partition functions}
\author{Tapas Bhowmik, Wei-Lun Tsai, and Dongxi Ye}

\address{Department of Mathematics, University of South Carolina, Columbia, SC 29208, USA}
\email{tbhowmik@email.sc.edu}
\email{weilun@mailbox.sc.edu}

\address{
School of Mathematics (Zhuhai), Sun Yat-sen University, Zhuhai 519082, Guangdong,
People's Republic of China}

\email{yedx3@mail.sysu.edu.cn}


\thanks{Wei-Lun Tsai was supported by the AMS-Simons Travel Grant. Dongxi Ye was supported by the Guangdong Basic and Applied Basic Research
Foundation (Grant No. 2024A1515030222).}

\begin{document}

\maketitle

\begin{abstract} We give Euler-like recursive formulas for the $t$-colored partition function when $t=2$ or $t=3,$ as well as for all $t$-regular partition functions. In particular, we derive an infinite family of ``triangular number" recurrences for the $3$-colored partition function. Our proofs are inspired by the recent work of Gomez, Ono, Saad, and Singh on the ordinary partition function and make extensive use of $q$-series identities for $(q;q)_{\infty}$ and $(q;q)_{\infty}^3.$
\end{abstract}

\vspace{0.10in}

\section{Introduction and statement of main results}


A \textit{partition} of a positive integer $n$ is a non-increasing sequence of positive integers 
$n_1 \geq n_2 \geq \dots \geq n_{\ell} > 0$ such that $n_1 + n_2 + \cdots + n_{\ell}=n$. 
Let $p(n)$ count the number of partitions of $n$. The celebrated  Euler's pentagonal number theorem gives a recurrence for computing $p(n)$ (see e.g., \cite[Corollary 1.8]{Andrews})
\begin{align}\label{Euler}
    p(n)=p(n-1) + p(n-2) -p(n-5) -p(n - 7) +\cdots=\sum_{j\in\mathbb{Z}\setminus\{0\}}(-1)^{j-1}p(n-w_j),
\end{align}
where $w_j:=\frac{3j^2+j}{2}$ is the $j^{\mathrm{th}}$ pentagonal number.

Recently, Gomez, Ono, Saad, and Singh \cite{GOSS} provided a beautiful framework for deriving a generalized form of the classical identity \eqref{Euler}, which can be viewed as a special case of their results. One notable application of their work offers an infinite family of ``pentagonal number'' recurrences for $p(n),$ expressed in terms of divisor functions and Hecke traces of an infinite sum of values of twisted quadratic Dirichlet series. In a similar spirit, this paper aims to derive Euler-type recurrence formulas for the $t$-colored partition function and the $t$-regular partition function.

A \textit{$t$-colored partition} is a partition in which each part can appear in one of $t$ colors. Let $\mathfrak{p}_t(n)$ denote the number of $t$-colored partitions of $n.$ A \textit{$t$-regular partition} is a partition where no part is a multiple of the positive integer $t.$ Let $p_t(n)$ denote the number of $t$-regular partitions of $n.$ By applying the $q$-series identities for $(q;q)_{\infty}$ and $(q;q)_{\infty}^3$, our first main result is to provide an Euler-like recurrence relation for the $2$-colored partition function.

\begin{theorem}\label{2-col-thm}
    For every positive integer $n$, we have
$$
\mathfrak{p}_2(n)=\begin{cases}
    (-1)^j+\sum_{k\in\mathbb{Z}_{\geq 1}} (-1)^{k+1}(2k+1)\,\mathfrak{p}_2(n-T_k) & \text{ if\,\, } n=w_j\\
\sum_{k\in\mathbb{Z}_{\geq 1}} (-1)^{k+1}(2k+1)\,\mathfrak{p}_2(n-T_k) & \text{\, otherwise\,\, } 
\end{cases},
$$
where $T_j:=\frac{j^2+j}{2}$ is the $j^{\mathrm{th}}$ triangular number.
\end{theorem}
We recall the Dedekind eta-function and its 
$q$-series expansion using Euler’s Pentagonal Number Theorem 
$$
\eta(z):=q^\frac{1}{24}\lp q;q\rp_\infty
=q^\frac{1}{24}\sum_{k\in\mathbb{Z}}(-1)^k q^{w_k}, \quad q:=e^{2\pi iz},
$$
where $z$ is in the upper half-plane $\mathbb{H}$ and $(q;q)_\infty:=\prod_{n=1}^\infty (1-q^n)$.

By the Jacobi Triple Product identity, we have 
\begin{equation}\label{eta3}
\eta(z)^3=q^\frac{1}{8}\lp q;q\rp_\infty ^3
=q^{\frac{1}{8}}\sum_{k\in\mathbb{Z}_{\geq 0 }} (-1)^{k}(2k+1) \,q^{T_k}.
\end{equation}
Next the generating function of $3$-colored partitions is given by
\begin{equation}\label{eta-3}
\sum_{n\geq 0}\mathfrak{p}_3(n)\, q^n=\frac{1}{\lp q;q\rp_\infty^3}
=\frac{q^{\frac{1}{8}}}{\eta(z)^3}.
\end{equation}
Substituting the $q$-expansions in
$$
\frac{1}{\eta(z)^3}\cdot\eta(z)^3=1
$$
and comparing the coefficients of $q^n$, it clearly follows that for $n\geq 1$,
\begin{equation}\label{col3v0}
\mathfrak{p}_3(n)=\sum_{k\in\mathbb{Z}_{\geq 1 }} (-1)^{k+1}(2k+1)\,\mathfrak{p}_3(n-T_k).
 \end{equation}
 
Our goal is to show that \eqref{col3v0} is indeed a glimpse of an infinite family of recurrence relations satisfied by $\mathfrak{p}_3(n)$. To this end, we first recall that for two smooth functions $f$ and $g$ on $\mathbb{H},$ and $k, l \in\mathbb{R}$ and $v\in\mathbb{Z}_{\geq0}$, the $v^{\mathrm{th}}$ {\it Rankin–Cohen
bracket} 
(see e.g., \cite{BFOR,Zagier}) is given as
\begin{align}\label{RCB}
[f,g]_v:=\sum_{\substack{r,s\geq0\\r+s=v}} \frac{(-1)^r\Gamma(k + v)\Gamma(l +v)}{s!\,r!\,\Gamma(k + v- s)\Gamma(l + v- r)}
\cdot D^r(f(z)) \cdot D^s(g(z)),    
\end{align}
where the differential operator is given by $D :=\frac{1}{2\pi i}\frac{d}{dz}=q\frac{d}{dq}.$

Moreover, for each $v\geq0$, with $k=-\frac{3}{2}$ and $l=\frac{3}{2},$ we define
\begin{equation}\label{Rv}
R_v(z):=\left[{1}/{\eta(z)^3},\eta(z)^3\right]_v.
\end{equation}
Differentiating the $q$-expansions in \eqref{eta3} and \eqref{eta-3} and replacing in \eqref{Rv}, we obtain
$$
R_v(z)=\sum_{n\geq 0} \sum_{k\geq 0}(-1)^k\,\mathcal{E}_v(n,k)\,\mathfrak{p}_3\lp n-T_k\rp q^n,
$$
where
\begin{equation}\label{evnk}
\mathcal{E}_v(n,k):=\sum_{\substack{r,s\geq0\\r+s=v}}\frac{(-1)^{r}}{8^v}\, \frac{\Gamma\left(v-\frac{3}{2}\right)\Gamma\left(v+\frac{3}{2}\right)}{s!\,r!\,\Gamma\left(r-\frac{3}{2}\right)\Gamma\left(s+\frac{3}{2}\right)}(2k+1)^{2s+1}\left(8n-(2k+1)^2\right)^{r}.
\end{equation}
For $v\geq 1$, we recall the weight $2v$ Eisenstein series 
$$
E_{2v}(z):=1-\frac{4v}{B_{2v}}\sum_{n=1}^\infty \sigma_{2v-1}(n)\, q^n,
$$
where  $B_{2v}$ is the $2v^{\mathrm{th}}$ Bernoulli number, and the divisor function $\sigma_{2v-1}(n):= \sum_{d|n}d^{2v-1}$. If $v\geq 2$, then $E_{2v}(z)$ is indeed a modular form of weight $2v$ and level $1$. Although $E_2(z)$ is not a modular form (see, e.g. \cite[Remark 1.20]{O}), it plays a significant role when studying the differential operator $D$ on modular forms.

We also require the level $1$ cuspidal Hecke eigenforms of weights 12, 16, 18,
20, 22, and 26
$$
\Delta_{2v}(z) := q +\sum_{n=2}^{\infty}\tau_{2v}(n)q^n \quad \mathrm{for}~v\in\{6,8,9,10,11,13\}.
$$
In particular, we have that $\Delta_{12}(z)=\Delta(z)$ is the Ramanujan Delta function.
In the following theorems, we demonstrate that \eqref{col3v0} represents the $v=0$ case of a rich family of ``triangular number'' recurrences satisfied by $\mathfrak{p}_3(n)$. For small values of $v,$ we provide simple explicit formulas.
\begin{theorem}\label{3-col-thm1}
\label{main} With the above setup the followings are true.

\noindent
(1) If $v\in\{2,3,4,5,7\}$, then for any positive integer $n$,
$$
\mathfrak{p}_3(n)=\frac{1}{\mathcal{E}_v(n,0)}\left(-\frac{4v\, \alpha_v}{B_{2v}}\,\sigma_{2v-1}(n)+\sum_{k\geq 1}\,(-1)^{k+1}\mathcal{E}_v(n,k)\,\mathfrak{p}_3\left(n-T_k\right) \right),
$$
where
the constant $\alpha_v$ is given in the following table.
\begin{table}[h!]
\renewcommand{\arraystretch}{1.75}
    \centering
    \begin{tabular}{|c|c|c|c|c|c|}
\hline \rule[-3mm]{0mm}{7mm}
$v$ & $2$ & $3$ & $4$ & $5$ & $7$  \\ \hline
$\alpha_v$ & $\frac{1}{64}$ & $\frac{1}{128}$ & $\frac{15}{4096}$ & $\frac{7}{4096}$ &  $\frac{99}{262144}$ \\ \hline
\end{tabular}
\vspace{0.2in}
    \caption{Values of $\alpha_v$ for various $v$}
    \label{tab:dim12_label'}
    \normalsize
\end{table}

\noindent
(2) If $v\in\{6,8,9,10,11, 13\}$, then we have 
$$
\mathfrak{p}_3(n)=\frac{1}{\mathcal{E}_v(n,0)}\left(-\frac{4v\, \alpha_v}{B_{2v}}\,\sigma_{2v-1}(n)+\beta_v\, \tau_{2v}(n)+\sum_{k\geq 1}\,(-1)^{k+1}\mathcal{E}_v(n,k)\,\mathfrak{p}_3\left(n-T_k\right) \right),
$$   
where the constants $\alpha_v$ and $\beta_v$ are given in the following table.

\begin{table}[h!]
\renewcommand{\arraystretch}{1.75}
    \centering
    \begin{tabular}{|c|c|c|c|c|c|c|}
\hline \rule[-3mm]{0mm}{7mm}
$v$  & $6$ &  $8$ & $9$ & $10$ & $11$ & $13$ \\ \hline
$\alpha_v$ &  $\frac{105}{131072}$  & $\frac{3003}{16777216}$ & $\frac{715}{8388608}$ & $\frac{21879}{536870912}$ & $\frac{20995}{1073741824}$ & $\frac{156009}{34359738368}$ \\ \hline
$\beta_v$  & $-\frac{51051}{22112}$  & $-\frac{9429849}{1851904}$ & $-\frac{324385347}{44919808}$ & $-\frac{328502311137}{22886612992}$ & $-\frac{318771027861}{10182066176}$ & $-\frac{162957690002835}{1379781312512}$ \\ \hline
\end{tabular}
\vspace{0.2in}
    \caption{Values of $\alpha_v$ and $\beta_v$ for various $v$}
    \label{tab:dim12_label}
    \normalsize
\end{table}
\end{theorem}
\begin{remark}
    In Lemma \ref{non-zero}, we show that $\mathcal{E}_v(n,0)$ does not vanish for all $v \geq 0$ and $n\geq1.$
\end{remark}






Now we extend Theorem \ref{3-col-thm1} to an infinite family of triangular number recurrences satisfied by $\mathfrak{p}_3(n)$. In order to state this result we require the following setup. Let $f(z)\in S_{2v}(1)$ be a cuspidal Hecke eigenform of weight~$2v$ and level $1$. Then consider its Fourier expansion at the cusp $[i\infty]$
\begin{align*}
    f(z)=q+\sum_{n=2}^{\infty}a_f(n)q^n.
\end{align*}
For $s\in\mathbb{C}$ with $\mathfrak{R}(s)\geq 2v+2$, we define the twisted Dirichlet series
\begin{align*}
    D(f;s):=\sum_{n=1}^{\infty}\frac{\left(\frac{-4}{n}\right)\,a_f\left(\frac
   {n^2-1}{8} \right)}{n^s},
\end{align*}
where $\left(\frac{-4}{\cdot}\right)$ is the Kronecker symbol. Next we recall that the Petersson norm
\[
||f|| := \iint_{\mathrm{SL}_2(\mathbb{Z})\backslash\mathbb{H}}|f|^2y^{2v}\frac{dxdy}{y^2}.
\]
For $v\geq 2$, $0\leq j\leq v-2$ and $m\geq 0$, we define
\begin{align}\label{etilde}
    \widetilde{\mathcal{E}}_v(j,m)&:=\frac{(-1)^j}{8^v}\cdot\frac{\Gamma\lp v-\frac{3}{2}\rp\Gamma\lp v+\frac{3}{2}\rp}{\Gamma\lp \frac{7}{2}\rp \Gamma\lp- \frac{3}{2}\rp}\nonumber \\
& \quad\quad\times \lp \frac{2}{\pi}\rp^{2v-1}\frac{\lp 2v+m-2 \rp!}{j!\, m!\, \lp 2v-j-2\rp!}\cdot \frac{(v-j-1)^{\overline{v}}\left(\frac{5}{2}\right)^{\overline{j}}}{\left(-\frac{3}{2}-j\right)^{\overline{v}} \left(\frac{7}{2}\right)^{\overline{j}}},
\end{align}
where we use the rising factorial notation
$$
(a)^{\overline{j}}:=\begin{cases}
    a(a+1)\cdots(a+j-1) & \text{ if } j\geq 1\\
    1 & \text{ if } j=0.
\end{cases}
$$
Then we define the special weighted sum
\begin{align}\label{wtsum}
\mathcal{D}_f:=\frac{1}{||f||}\,\sum_{j=0}^{v-2}\sum_{m=0}^\infty \widetilde{\mathcal{E}}_v(j,m)\cdot D(f;2v+2j+2m+2),
\end{align}
and the {\it weight $2v$ Hecke trace} by
\begin{align}
\mathrm{Tr}_{2v}(n):=\sum_{f\in\mathfrak{B}_{2v}}\mathcal{D}_f\cdot a_f(n),
\end{align}
where $\mathfrak{B}_{2v}$ consists of orthogonal basis of normalized Hecke eigenforms of $S_{2v}(1)$.

 \begin{theorem}\label{3-col-thm2}
     For $v\in\mathbb{Z}_{\geq 8}\cup\{6\}$ and every positive integer $n$, we have
$$
\mathfrak{p}_3(n)=\frac{1}{\mathcal{E}_v(n,0)}\left(-\frac{4v\,\mathcal{E}_v(0,0)}{B_{2v}}\,\sigma_{2v-1}(n)+ \mathrm{Tr}_{2v}(n)
+\sum_{k\geq 1}\,(-1)^{k+1}\mathcal{E}_v(n,k)\,\mathfrak{p}_3\left(n-T_k\right) \right).
 $$ 
 \end{theorem}
\begin{remark} Readers interested in other types of elegant recurrences for $t$-colored partition functions may refer to the works of Choliy, Kolitsch, and Sills \cite{CKS}, Gandhi \cite{G}, Lazarev and Swisher \cite{LS}, and Sellers \cite{Sellers}.
\end{remark}

\noindent
\textbf{Example.}
    For $v=6$, we give a numerical justification of Theorem \ref{3-col-thm2}. Comparing the expressions of $\mathfrak{p}_3(n)$ in Theorems \ref{3-col-thm1} and \ref{3-col-thm2}, and then using Table \ref{tab:dim12_label}, we get that
$$
\mathcal{D}_\Delta=\beta_{6}=-2.308746\ldots.
$$
Next we compute an estimated value of $\mathcal{D}_\Delta$ via the infinite sum in \eqref{wtsum}. If we consider
$$
\begin{aligned}
D(\Delta;N,s)&:=\sum_{n=1}^{N}\frac{\left(\frac{-4}{n}\right)\,\tau\left(\frac
   {n^2-1}{8} \right)}{n^s},\\
\mathcal{D}_\Delta(M,N)&:=\frac{1}{||\Delta||}\,\sum_{j=0}^{4}\sum_{m=0}^M \widetilde{\mathcal{E}}_6(j,m)\cdot D(\Delta;N,2j+2m+14),
\end{aligned}
 $$ 
then a fast computation in \texttt{Mathematica} gives us $\mathcal{D}_\Delta(100,700)=-2.308746\ldots .$\\

Finally, we turn to the $t$-regular partition function. Recall the generating function identity
\begin{align}\label{t-reg-q}
\sum_{n=0}^\infty p_t(n)\,q^n=\frac{\left(q^t;q^t\right)_\infty}{(q;q)_\infty}.
\end{align}
Using \eqref{t-reg-q} and a similar argument as in the proof of Theorem \ref{2-col-thm} yields the following recurrence formula for all $t$-regular partition functions.
\begin{theorem}\label{t-reg-thm}
  For fixed $t\in\mathbb{Z}_{\geq 2}$ and $n\in\mathbb{Z}_{\geq 1}$, we have
$$
p_t(n)=\begin{cases}
    (-1)^j+\sum_{k\in\mathbb{Z}\setminus\{0\}} (-1)^{k+1}\,p_t(n-w_k) & \text{ if\,\, } n=t w_j\\
\sum_{k\in\mathbb{Z}\setminus\{0\}} (-1)^{k+1}\,p_t(n-w_k) & \text{\, otherwise\,\, } 
\end{cases}.
$$
\end{theorem}

\section*{Acknowledgments} We thank Ken Ono for his valuable insights and suggestions during this project. We also thank Bruce Sagan for his interest and for pointing out the omission of $2k+1$ in the previous version of Theorem \ref{2-col-thm}. The second author thanks Matt Boylan for his helpful conversations during the Palmetto Number Theory Series conference, PANTS XXXIX.

\section{Preliminaries}
\begin{lemma}\label{non-zero}
 For $v\geq 0$, we have that $\mathcal{E}_v(n,0)$ is nonzero for every positive integer $n$. 
\end{lemma}
\begin{proof}
From \eqref{evnk}, we have that
$$
\mathcal{E}_v(n,0)=\sum_{\substack{r,s\geq0\\r+s=v}}\frac{(-1)^{r}}{8^v}\, \frac{\Gamma\left(v-\frac{3}{2}\right)\Gamma\left(v+\frac{3}{2}\right)}{r! \,s!\,\Gamma\left(r-\frac{3}{2}\right)\Gamma\left(s+\frac{3}{2}\right)}\,\left(8n-1\right)^{r}.
$$
Substituting the values of gamma function at half-integer points and after simplifying, we obtain
$$
\begin{aligned}
   \mathcal{E}_v(n,0)&=\frac{\Gamma\left(v-\frac{3}{2}\right)\Gamma\left(v+\frac{3}{2}\right)}{2^{v+1}\,\pi}\sum_{r=0}^v\,(-1)^r \frac{\lp 2r-1\rp \lp 2r-3\rp}{(2r)! \lp 2v+1-2r\rp!}\, \lp 8n-1\rp^r\\
   &=\frac{\Gamma\left(v-\frac{3}{2}\right)\Gamma\left(v+\frac{3}{2}\right)}{\lp 2v+1\rp! \,2^{v+1}\,\pi}\cdot \mathcal{P}_v(8n-1),
\end{aligned}
$$
where
$$
\mathcal{P}_v(n):=\sum_{r=0}^v\, (-1)^r\binom{2v+1}{2r} (2r-3)(2r-1)\, n^r
$$
 is a polynomial with integer coefficients and the constant term $3$. Note that $8n-1>1$ for all $n\geq 1$, and by induction argument it follows that there always exists a prime $l$ other than $3$ such that $l|(8n-1)$. Thus we get $\mathcal{P}_v(8n-1)\equiv 3\pmod{l}$, which implies that $\mathcal{E}_v(n,0)$ is non-zero for all $n\geq 1$.
\end{proof}
We now recall the modular transformation property of the Dedekind eta-function
$$
\eta\lp\gamma(z) \rp=\varepsilon(\gamma) (cz+d)^\frac{1}{2}\eta(z) \text{\,\, for\,\,}\gamma=\lp\begin{smallmatrix}
    a& b\\c & d
\end{smallmatrix}\rp\in \mathrm{SL}_2(\mathbb{Z}),
$$
where the multiplier $\varepsilon(\gamma)$ is given by the following formulas:
\begin{equation}\label{multiplier}
\varepsilon(\gamma):=\begin{cases}
    \lp \frac{d}{|c|}\rp\exp\lp \frac{\pi i}{12}\lp c(a+d-3)-bd(c^2-1)\rp \rp & \text{ if \,\,}2\nmid c,\\
   \lp \frac{c}{|d|}\rp\exp\lp \frac{\pi i}{12}\lp c(a-2d)-bd(c^2-1)+3d-3\rp \rp\varepsilon(c,d) & \text{ if \,\,}2\mid c, 
\end{cases}
\end{equation}
where $\lp \frac{c}{d}\rp$ is the Kronecker symbol and $\varepsilon(c,d)=-1$ when $c\leq 0 $ and $d<0$, and $\varepsilon(c,d)=1$ otherwise.

Next we prove an important lemma regarding the modularity of $R_v(z)$, defined in \eqref{Rv}.


\begin{lemma}\label{RC-lemma}
    The Rankin--Cohen bracket $R_v(z):=[\eta(z)^{3},{1}/{\eta(z)^{3}}]_{v}$ is a holomorphic modular form of weight~$2v$ and level $1$.
\end{lemma}
\begin{proof}
    The proof is basically paralleled to that of \cite[Theorem~1.1]{GOSS} but uses the relations
    $$
    D\left(\eta(z)^{3}\right)=\frac{1}{8}E_{2}(z)\eta(z)^{3}\quad\mbox{and}\quad D\left(\frac{1}{\eta(z)^{3}}\right)=-\frac{1}{8}E_{2}(z)\frac{1}{\eta(z)^{3}}.
    $$
\end{proof}
We also require some general properties of the Rankin--Cohen bracket.
\begin{proposition}\label{RC-prop}
Suppose that $f$ and $g$ are modular forms on a subgroup $\Gamma$ of $\mathrm{SL}_2(\mathbb{Z})$ with multiplier systems $\varepsilon_f$ and $\varepsilon_g$ respectively. Then the followings are true.

\noindent
(1) We have that $\left[f,g\right]_v$ is modular of weight $k+l+2v$ with multiplier system $\varepsilon_f\varepsilon_g$ on $\Gamma$.

\noindent
(2) Under the modular slash operator, we have
$$
\left[f|_k\gamma,\, g|_l\gamma\right]_v=[f,g]_v\big|_{k+l+2v}\gamma \text{ \quad for all \,}\gamma\in \mathrm{SL}_2(\mathbb{Z}).
$$   
\end{proposition}
\begin{proof}
See, e.g. \cite[Proposition 2.37]{BFOR} for details.   
\end{proof}
Then we need to introduce certain harmonic Maass-Poincar\'e series of negative half-integer weight $k$. Suppose that $s\in\mathbb{C}$ with $\mathfrak{R}(s)>1$, $y\in\mathbb{R}\setminus\{0\}$, and let
\begin{align*}
\textbf{M}_{k,s}(y):= {|y|}^{-\frac{k}{2}}M_{\operatorname{sign}(y)\frac{k}{2},s-\frac{1}{2}}(|y|)
\end{align*}
with $M_{\lambda,\mu}$ denotes the usual $M$-Whittaker function (see \cite[Section 13.14]{DLMF}). For convenience, we also define
$$
\phi_s(z):=\textbf{M}_{k,s}( 4\pi y)\, e(x),
$$
where $z=x+iy$ and $e(x):=e^{2\pi ix}$. Suppose that $\Gamma$ is a congruence subgroup of $\mathrm{SL}_2(\mathbb{Z})$ and $[\rho]$ is the cusp satisfying $\rho=\gamma_\rho^{-1}(i\infty)$. If $\Gamma_\rho$ denotes the stabilizer of $[\rho]$ in $\Gamma$ and $\nu$ is a multiplier system, then we define the Poincar\'e series 
$$
P_{[\rho]}\lp z,m,k,s,\Gamma,\nu\rp:=\frac{1}{\Gamma(2-k)}\sum_{\gamma\in \Gamma_\rho\backslash\Gamma}\, \phi_s\lp \lp \frac{-m+\kappa}{t} \rp \gamma_\rho \,z\rp\bigg|_{k,\nu} \gamma,
$$
where $t$ and $\kappa$ are the cusp width and the cusp parameter respectively.

From now on we can restrict to $\Gamma:=\mathrm{SL}_2(\mathbb{Z})$, $\rho:=i\infty$, and $\gamma_\rho:=I_{2\times 2}$. It is known that for $k<0$ and $s=\frac{2-k}{2}$, the Poincar\'e series $P_{[\rho]}\lp z,m,k,\nu\rp$ is indeed a harmonic Maass form of weight $k$ and multiplier $\nu$ on $\Gamma$ (see \cite[Lemma 3.1]{BO}). In the following lemma, we identify $1/\eta(z)^3$ as a harmonic Maass form in terms of Poincar\'e series.

\begin{lemma}\label{main-lemma} With the above setup we have
\begin{align}
    \frac{1}{\eta(z)^3}=P_{[i\infty]}(z,8,-3/2,\overline{\varepsilon}^{3}),
\end{align}
where $\varepsilon$ is the multiplier system defined as in \eqref{multiplier}.
\end{lemma}

\begin{proof}
    The argument for this lemma is essentially the same as that of \cite[Proposition~3.4]{GOSS} by showing that both sides viewed as harmonoic Maass forms have identical principal parts at the cusp~$[i\infty]$.
\end{proof}

\begin{proposition}\label{main-prop} Suppose that $f(z)\in S_{2v}(1)$ has real Fourier coefficients. Then the weighted sum $\mathcal{D}_f$ defined as in \eqref{wtsum} can be expressed in terms of the following Petersson inner product 
    \begin{align}
    ||f||\cdot\mathcal{D}_f= \left\langle[1/\eta(z)^3, \eta(z)^3]_{v},\,f(z)\right\rangle.
    \end{align}
\end{proposition}
\begin{proof}
From Lemma \ref{main-lemma}, we have that
$$
\frac{1}{\eta(z)^3}=P_{[i\infty]}(z,8,-{3}/{2},\overline{\varepsilon}^{3})=\frac{1}{\Gamma\lp \frac{7}{2}\rp}\sum_{\gamma\in\Gamma_{i\infty} \backslash\Gamma}\phi_{\frac{7}{4}}\lp -\frac{z}{8}\rp\Bigg|_{-\frac{3}{2},\,\overline{\varepsilon}^3}\gamma.
$$
For convenience, we let
$$
\delta(z):=\phi_{\frac{7}{4}}\lp -\frac{z}{8}\rp=e\lp -\frac{x}{8}\rp \cdot\lp \frac{\pi y}{2}\rp^\frac{3}{4} M_{\frac{3}{4},\frac{5}{4}}\lp \frac{\pi y}{2}\rp.
$$
Then we consider
$$
\begin{aligned}
\left[1/\eta^3,\eta^3\right]_v&= \frac{1}{\Gamma\lp \frac{7}{2}\rp}\sum_{\gamma\in\Gamma_{i\infty} \backslash\Gamma}\Bigg[\delta(z)\Big|_{-\frac{3}{2},\,\overline{\varepsilon}^3}\gamma,\,\eta(z)^3\Bigg]_v  \\
&=\frac{1}{\Gamma\lp \frac{7}{2}\rp}\sum_{\gamma\in\Gamma_{i\infty} \backslash\Gamma}\Bigg[\delta(z)\Big|_{-\frac{3}{2},\,\overline{\varepsilon}^3}\gamma,\,\eta(z)^3\Big|_{\frac{3}{2},\,{\varepsilon}^3}\gamma^{-1}\Big|_{\frac{3}{2},\,{\varepsilon}^3}\gamma\Bigg]_v\\
&=\frac{1}{\Gamma\lp \frac{7}{2}\rp}\sum_{\gamma\in\Gamma_{i\infty} \backslash\Gamma}\big[\delta(z),\,\eta(z)^3\big]_v\Bigg|_{2v}\gamma,
\end{aligned}
$$
where the last two steps require the modularity of $\eta(z)^3$ and Proposition~\ref{RC-prop} respectively.

Using the modularity of $f(z)$, we get from the definition of the Peterson inner product
$$
\begin{aligned}
    \left\langle \left[1/\eta^3,\eta^3\right]_v,f\right\rangle&=\frac{1}{\Gamma\lp \frac{7}{2}\rp}\iint_{\Gamma
\backslash\mathbb{H}} \,\,\sum_{\gamma\in\Gamma_{i\infty} \backslash\Gamma}\big[\delta(z),\,\eta(z)^3\big]_v\Bigg|_{2v}\gamma\,\cdot \overline{f}(x+iy)\, y^{2v} \frac{dx dy}{y^2}\\
    &=\frac{1}{\Gamma\lp \frac{7}{2}\rp}\sum_{\gamma\in\Gamma_{i\infty} \backslash\Gamma}\,\iint_{\gamma^{-1}\lp\Gamma
    \backslash\mathbb{H}\rp} \,\,\sum_{\gamma\in\Gamma_{i\infty} \backslash\Gamma}\big[\delta(z),\,\eta(z)^3\big]_v\cdot \overline{f}(x+iy)\, y^{2v} \frac{dx dy}{y^2}.
\end{aligned}
$$
Finally, using the fact that $\Gamma_{i\infty}\backslash \mathbb{H}=\cup_{\gamma\in\Gamma_{i\infty} \backslash\Gamma}\,\,\, \gamma^{-1}\lp\Gamma
\backslash\mathbb{H}\rp$, we obtain
\begin{equation}\label{L1-eqn}
    \left\langle \left[1/\eta^3,\eta^3\right]_v,f\right\rangle=\frac{1}{\Gamma\lp \frac{7}{2}\rp}\int_0^\infty \int_0^1 \left[\delta(z),\, \eta(z)^3\right]_v\,\overline{f}(x+iy)\, y^{2v}\frac{dx dy}{y^2}.
\end{equation}

To compute the double integral in \eqref{L1-eqn}, we first make the integrand explicit. From the definition in \eqref{RCB}, we have that
\begin{equation}\label{L1-eqn'}
\left[\delta(z),\,\eta(z)^3\right]_v=\sum_{\substack{r,s\geq0\\r+s=v}} (-1)^r\, c_{r,s}
\, D^r(\delta(z)) \, D^s(\eta(z)^3),
\end{equation}
where we define for simplicity
$$
c_{r,s}:=\frac{\Gamma\lp v-\frac{3}{2}\rp\Gamma\lp v+\frac{3}{2}\rp}{s!\,r!\,\Gamma\lp r-\frac{3}{2}\rp\Gamma\lp s+\frac{3}{2}\rp}.
$$
From \eqref{L1-eqn'} it clear that we need to compute $D^r(\delta(z))$, which is performed in the following lemma.
\begin{lemma}\label{L1-lemma}
    For $r\geq 0$, we have
$$
D^r\lp \delta(z)\rp=\frac{1}{8^r}\lp \frac{-5}{2}\rp^{\overline{r}}e\lp\frac{-x}{8}\rp\lp \frac{\pi y}{2}\rp^{\frac{3}{4}-\frac{r}{2}} M_{\frac{3}{4}-\frac{r}{2},\frac{5}{4}-\frac{r}{2}}\lp \frac{\pi y}{2}\rp. 
$$
\end{lemma}
\begin{proof}
 This can be easily proved via the derivative identity \cite[(13.15.15)]{DLMF} of $M$-Whittaker function and Leibniz rule. We refer to \cite[Lemmas 3.6 and 3.7]{GOSS} for details of a similar calculation. 
\end{proof}
From the classical identity $\eta(z)^3=\sum_{n\geq 1} \lp \frac{-4}{n}\rp\, n q^{n^2/8}$ (see, e.g. \cite[p. 17]{O}), we find that
\begin{equation}\label{L1-eqn''}
    D^s\lp \eta(z)^3\rp=\frac{1}{8^s}\sum_{n\geq 1}\lp\frac{-4}{n}\rp\, n^{2s+1}  q^{\frac{n^2}{8}}
\end{equation}
Applying \eqref{L1-eqn''} and Lemma \ref{L1-lemma} together in \eqref{L1-eqn'}, we obtain 
$$
\begin{aligned}
\left[\delta(z),\,\eta(z)^3\right]_v=\frac{1}{8^v}\sum_{\substack{r,s\geq0\\r+s=v}} (-1)^r\, c_{r,s}
\,\lp \frac{-5}{2}\rp^{\overline{r}} & e\lp\frac{-x}{8}\rp\lp \frac{\pi y}{2}\rp^{\frac{3}{4}-\frac{r}{2}}\\
&\times  M_{\frac{3}{4}-\frac{r}{2},\frac{5}{4}-\frac{r}{2}}\lp \frac{\pi y}{2}\rp
\sum_{n\geq 1}\lp\frac{-4}{n}\rp\, n^{2s+1}  q^{\frac{n^2}{8}}.
\end{aligned}
$$
Substituting this in \eqref{L1-eqn} and by assumption that $f(z)$ has real Fourier coefficients, we obtain \small
$$
\begin{aligned}
\left\langle \left[1/\eta^3,\eta^3\right]_v,f\right\rangle &=\frac{1}{8^v\,\Gamma\lp \frac{7}{2}\rp} \sum_{\substack{r,s\geq0\\r+s=v}} (-1)^r\, c_{r,s}
\,\lp \frac{-5}{2}\rp^{\overline{r}}\sum_{n\geq 1}\sum_{m\geq 1}\lp\frac{-4}{n}\rp\, n^{2s+1} a_f(m) 
\\
\times\int_0^\infty \int_{0}^ 1 &\lp \frac{\pi y}{2}\rp^{\frac{3}{4}-\frac{r}{2}} M_{\frac{3}{4}-\frac{r}{2},\frac{5}{4}-\frac{r}{2}}\lp \frac{\pi y}{2}\rp \exp{\lp-\frac{\pi y (n^2+8m)}{4}\rp} e{\lp\frac{(n^2-1-8m)x}{8}\rp } y^{2v}\, \frac{dx dy}{y^2}.
\end{aligned}
$$\normalsize
Note that if $(n,4)=1$, then $8|(n^2-1)$, and hence 
$$
\int_0^1 e{\lp\frac{(n^2-1-8m)x}{8}\rp }\, dx=\begin{cases}
    1 & \text{ if } m=\frac{n^2-1}{8}\\
    0 & \text{ otherwise }
\end{cases}.
$$
Thus we obtain \small
\begin{equation}\label{L3-eqn}
    \left\langle \left[1/\eta^3,\eta^3\right]_v,f\right\rangle =\frac{1}{8^v\,\Gamma\lp \frac{7}{2}\rp} \sum_{\substack{r,s\geq0\\r+s=v}} (-1)^r c_{r,s}
\lp \frac{-5}{2}\rp^{\overline{r}}\sum_{n> 1}\lp\frac{-4}{n}\rp n^{2s+1} a_f\lp \frac{n^2-1}{8}\rp \mathcal{I}(r,n),
\end{equation}\normalsize
where
\begin{equation}\label{L3-eqn'}
    \mathcal{I}(r,n):=\int_0^\infty \lp \frac{\pi y}{2}\rp^{\frac{3}{4}-\frac{r}{2}} M_{\frac{3}{4}-\frac{r}{2},\frac{5}{4}-\frac{r}{2}}\lp \frac{\pi y}{2}\rp \exp{\lp-\frac{\pi (2n^2-1)y}{4}\rp} y^{2v-2}\, dy.
\end{equation}
The following lemma gives the evaluation of $\mathcal{I}(r,n)$ in terms of special functions.
\begin{lemma}\label{L3-lemma}
    If $n>1$ and $0\leq r\leq v$, then we have
$$
\mathcal{I}(r,n)=\lp \frac{2}{\pi}\rp^{2v-1}\cdot\frac{\Gamma\lp 2v-r+\frac{3}{2}\rp}{n^{4v-2r+3}} \cdot {}_2F_1 \left(\begin{matrix}
    1& \frac{3}{2}+2v-r& \\
    {} & \frac{7}{2}-r
\end{matrix}|\,\frac{1}{n^2}\right).
$$    
\end{lemma}
\begin{proof}
Making the change of variable $t=\frac{\pi y}{2}$, we get
$$
    \mathcal{I}(r,n)=\lp\frac{2}{\pi}\rp^{2v-1}\int_0^\infty t^{2v-\frac{r}{2}-\frac{5}{4}} M_{\frac{3}{4}-\frac{r}{2},\frac{5}{4}-\frac{r}{2}}\lp t\rp \exp{\lp-\lp n^2-\frac{1}{2} \rp t\rp} dt
$$
Then we evaluate this integral using the identity \cite[(13.23.1)]{DLMF}, which proves the lemma.
\end{proof}
Combining \eqref{L3-eqn} with Lemma \ref{L3-lemma}, we deduce that
\begin{equation}\label{L4-eqn}
     \left\langle \left[1/\eta^3,\eta^3\right]_v,f\right\rangle =\frac{1}{8^v\,\Gamma\lp \frac{7}{2}\rp}\cdot\lp \frac{2}{\pi}\rp^{2v-1}\sum_{n>1} 
     \frac{\lp \frac{-4}{n}\rp \,a_f\lp \frac{n^2-1}{8}\rp}{n^{2v+2}}\cdot \omega_v(n),
\end{equation}
where
\begin{equation}\label{L4-eqn'}
\omega_v(n):=\sum_{r=0}^v (-1)^r\, c_{r,v-r}
\, \lp -\frac{5}{2}\rp^{\overline{r}}\Gamma\lp \frac{3}{2}+2v-r\rp\cdot\, {}_2F_1 \left(\begin{matrix}
    1& \frac{3}{2}+2v-r& \\
    {} & \frac{7}{2}-r
\end{matrix}|\,\frac{1}{n^2}\right).
\end{equation}
Now we aim to express $\omega_v(n)$ in a different form that would be the source of infinite sum of special values of twisted Dirichlet series. For this purpose, we need the following lemma.
\begin{lemma}\label{L5-lemma}
For $v\geq 1$ and $0\leq r\leq v$, we have the following identities:
$$
 {}_2F_1 \left(\begin{matrix}
    1& \frac{3}{2}+2v-r& \\
    {} & \frac{7}{2}-r
\end{matrix}|\,\frac{1}{n^2}\right)=\lp 1-\frac{1}{n^2} \rp^{1-2v}\,\,\sum_{i=0}^{2v-2} \frac{(-1)^i}{n^{2i}} \binom{2v-2}{i}\frac{\lp \frac{5}{2}-r \rp^{\overline{i}}}{\lp \frac{7}{2}-r \rp^{\overline{i}}},
 $$ 
$$
\begin{aligned}
\sum_{i=0}^{2v-2} \frac{(-1)^i}{n^{2i}} \binom{2v-2}{i} &\sum_{r=0}^v (-1)^r\, c_{r,v-r}
\, \lp -\frac{5}{2}\rp^{\overline{r}}\Gamma\lp \frac{3}{2}+2v-r\rp 
 \frac{\lp \frac{5}{2}-r \rp^{\overline{i}}}{\lp \frac{7}{2}-r \rp^{\overline{i}}}\\
 &=\frac{\Gamma\lp v-\frac{3}{2}\rp \Gamma\lp v+\frac{3}{2}\rp}{\Gamma\lp -\frac{3}{2}\rp}\sum_{i=0}^{v-2} \frac{(-1)^i}{n^{2i}} \binom{2v-2}{i}\frac{(v-i-1)^{\overline{v}}\lp\frac{5}{2} \rp^{\overline{i}}}{\lp -\frac{3}{2}-i\rp^{\overline{v}}\lp \frac{7}{2}\rp^{\overline{i}}}.
\end{aligned}
$$ 
\end{lemma}
\begin{proof}
    The proof mainly uses two classical identities for hypergeometric functions due to Euler (see \cite[Theorem~2.2.2]{Special-Functions}) and Pfaff-Saalsch\"utz (see \cite[Theorem 2.2.6]{Special-Functions}). As it can be easily adopted from \cite[Lemmas 3.12 and 3.13]{GOSS}, we leave all the details here.
\end{proof}
Applying the identities from Lemma \ref{L5-lemma} in \eqref{L4-eqn'} successively, we obtain
$$
\begin{aligned}
\omega_v(n)&=\frac{\Gamma\lp v-\frac{3}{2}\rp \Gamma\lp v+\frac{3}{2}\rp}{\Gamma\lp -\frac{3}{2}\rp}\cdot \frac{1}{\lp 1-\frac{1}{n^2}\rp^{2v-1}}\sum_{i=0}^{v-2}\frac{(-1)^i}{n^{2i}} \binom{2v-2}{i}\frac{(v-i-1)^{\overline{v}}\lp\frac{5}{2} \rp^{\overline{i}}}{\lp -\frac{3}{2}-i\rp^{\overline{v}}\lp \frac{7}{2}\rp^{\overline{i}}}\\
&= \frac{\Gamma\lp v-\frac{3}{2}\rp \Gamma\lp v+\frac{3}{2}\rp}{\Gamma\lp -\frac{3}{2}\rp}\sum_{i=0}^{v-2}\sum_{m=0}^\infty\,\frac{(-1)^i}{n^{2i+2m}} \binom{2v-2}{i}\binom{2v+m-2}{m}\frac{(v-i-1)^{\overline{v}}\lp\frac{5}{2} \rp^{\overline{i}}}{\lp -\frac{3}{2}-i\rp^{\overline{v}}\lp \frac{7}{2}\rp^{\overline{i}}}.
\end{aligned}
$$
Finally, substituting the above expression of $\omega_v(n)$ in \eqref{L4-eqn}, we obtain
$$
\begin{aligned}
\left\langle \left[1/\eta^3,\eta^3\right]_v,f\right\rangle &=\sum_{i=0}^{v-2}\sum_{m=0}^\infty \widetilde{\mathcal{E}}_v(i,m)\sum_{n>1} \frac{\lp \frac{-4}{n}\rp\, a_f\lp \frac{n^2-1}{8}\rp}{n^{2v+2i+2m+2}}\\
&= ||f||\cdot \mathcal{D}_f,
\end{aligned}
$$
where $\widetilde{\mathcal{E}}_v(i,m)$ and $\mathcal{D}_f$ are the same as defined in \eqref{etilde} and \eqref{wtsum} respectively. This completes the proof of Proposition \ref{main-prop}.
\end{proof}
\section{Proofs of Theorems \ref{2-col-thm} and \ref{t-reg-thm}}
\begin{proof}[\bf Proof of Theorem \ref{2-col-thm}]
From the generating function of $2$-colored partitions, we have that
$$
\sum_{n\geq 0} \mathfrak{p}_2(n)\, q^n=\frac{1}{(q;q)_\infty^2}=\frac{(q;q)_\infty}{(q;q)_\infty^3}.
$$
Then we recall the following classical identities
$$
(q;q)_\infty=\sum_{j\in\mathbb{Z}} (-1)^j q^{w_j}, \,\quad (q;q)_\infty^3=\sum_{k\in\mathbb{Z}_{\geq 0}} (-1)^k (2k+1)\, q^{T_k}.
$$
Thus we get
\begin{equation}\label{2-col-proof-eq1}
\left(\sum_{n\geq 0} \mathfrak{p}_2(n)\, q^n \right)\left( \sum_{k\in\mathbb{Z}_{\geq 0}} (-1)^k (2k+1)\, q^{T_k}\right)=\sum_{j\in\mathbb{Z}} (-1)^j q^{w_j}.
\end{equation}
Note that the coefficient of $q^n$ on the left-hand side of \eqref{2-col-proof-eq1} is
$$
\mathfrak{p}_2(n)+\sum_{k\in\mathbb{Z}_{\geq 1}} (-1)^{k}(2k+1)\,\mathfrak{p}_2(n-T_k)
$$
and on the other side is equals to $$\begin{cases}
 (-1)^j & \text{if\,\, } n= w_j\\
 0     & \text{ otherwise}. 
\end{cases}$$  
This immediately gives the recursive formula of $\mathfrak{p}_2(n)$.
\end{proof}

\begin{proof}[\bf Proof of Theorem \ref{t-reg-thm}]
    After rewriting the generating function of $p_t(n)$ and using the $q$-series expansion of $(q;q)_\infty$, we get
\begin{equation}\label{reg-proof-eq1}
\left(\sum_{n\geq 0} p_t(n)\, q^n \right)\left( \sum_{k\in\mathbb{Z}} (-1)^k \, q^{w_k}\right)=\sum_{j\in\mathbb{Z}} (-1)^j q^{t w_j}.
\end{equation}
Then the coefficient of $q^n$ on the left-hand side of \eqref{reg-proof-eq1} is given by 
$$
p_t(n)+\sum_{k\in\mathbb{Z}\setminus\{0\}} (-1)^k\, p_t(n-w_k)
 $$
 and on the right-hand side is
 $$\begin{cases}
 (-1)^j & \text{if\,\, } n=t w_j\\
 0     & \text{ otherwise}.
 \end{cases}$$
 Thus, comparing the coefficients on both sides, we obtain the desired recursive formula.
\end{proof}

\section{Proofs of Theorems \ref{3-col-thm1} and \ref{3-col-thm2}}
\begin{proof}[\bf Proof of Theorem \ref{3-col-thm1}]
Since $R_v(z)=\left[1/\eta(z)^3,\,\eta(z)^3\right]_v$, by Lemma \ref{RC-lemma} we get that $R_v(z)$ is a modular form of weight $2v$ and level $1$. For $v\in\{2,3,4,5,7\}$, we have $S_{2v}(1)=\{0\}$, and hence $R_v(z)=\alpha_v \,E_{2v}(z)$ for some constant $\alpha_v\in \mathbb{C}$. Considering the $q$-expansion on both sides and comparing the constant terms, we obtain $\alpha_v=\mathcal{E}_v(0,0)$. The claimed recurrence in part (1) follows after comparing the coefficient of $q^n$ and rearranging the terms.

For $v\in\{6,8,9,10,11,13\}$, we have that $\Delta_{2v}(z)\in S_{2v}(1)$ and $\dim S_{2v}(1)=1$. Thus we get $R_{v}(z)=\alpha_v\, E_{2v}(z)+\beta_v\, \Delta_{2v}(z)$ for some constants $\alpha_v,\beta_v\in\mathbb{C}$. Substituting the $q$-expansions on both sides we see that
$$
\alpha_v=\mathcal{E}_v(0,0),\,\text{ and }\, \beta_v=\frac{4v\, \mathcal{E}_v(0,0)}{B_{2v}}+3\mathcal{E}_v(1,0)-\mathcal{E}_v(1,1).
$$
Then we obtain the claimed recurrence in part (2) by comparing the coefficients of $q^n$.
\end{proof}
\begin{proof}[\bf Proof of Theorem \ref{3-col-thm2}]
For $v\in\mathbb{Z}_{\geq 8}\cup \{6\}$, the space $S_{2v}(1)$ is non-empty. Recall that $\mathfrak{B}_{2v}$ consists of orthogonal basis of normalized Hecke eigenforms of $S_{2v}(1)$. Thus we get that
$$
R_v(z)=\alpha_v\, E_{2v}(z)+\sum_{f\in \mathfrak{B}_{2v}}\frac{\left\langle \left[1/\eta^3,\eta^3\right]_v,\,f\right\rangle}{||f||}\, f(z).
$$
Since Hecke eigenforms have real Fourier coefficients, we apply Proposition \ref{main-prop} so that 
$$
R_v(z)=\alpha_v\, E_{2v}(z)+\sum_{f\in \mathfrak{B}_{2v}}\mathcal{D}_f\cdot f(z).
$$
Then we compare the coefficients of $q^n$ and simply solve for $\mathfrak{p}_3(n)$ to obtain the desired recurrence relations.  
\end{proof}

\end{document}